\documentclass[a4paper]{article}

\usepackage{amsfonts}
\usepackage{graphicx}
\usepackage{amsthm}
\usepackage{amsmath}
\usepackage{amssymb}
\usepackage{enumerate}
\allowdisplaybreaks

\newtheorem{proposition}{Proposition}[section]
\newtheorem{lemma}[proposition]{Lemma}

\newtheorem{theorem}[proposition]{Theorem}

\theoremstyle{definition}

\numberwithin{equation}{section}

\begin{document}

\begin{center}
\LARGE
\textbf{Number of directions determined by a set in $\mathbb{F}_{q}^{2}$ and growth in $\mathrm{Aff}(\mathbb{F}_{q})$}
\bigskip\bigskip

\large
Daniele Dona\footnote{The author was partially supported by the European Research Council under Programme H2020-EU.1.1., ERC Grant ID: 648329 (codename GRANT).}
\bigskip

\normalsize
Mathematisches Institut, Georg-August-Universit\"at G\"ottingen

Bunsenstra\ss e 3-5, 37073 G\"ottingen, Germany

\texttt{daniele.dona@mathematik.uni-goettingen.de}
\bigskip\bigskip\bigskip
\end{center}

\begin{minipage}{110mm}
\small
\textbf{Abstract.} We prove that a set $A$ of at most $q$ non-collinear points in the finite plane $\mathbb{F}_{q}^{2}$ spans at least $\approx\frac{|A|}{\sqrt{q}}$ directions: this is based on a lower bound contained in \cite{FST13}, which we prove again together with a different upper bound than the one given therein. Then, following the procedure used in \cite{RS18}, we prove a new structural theorem about slowly growing sets in $\mathrm{Aff}(\mathbb{F}_{q})$ for any finite field $\mathbb{F}_{q}$, generalizing the analogous results in \cite{Hel15} \cite{Mur17} \cite{RS18} over prime fields.
\medskip

\textbf{Keywords.} Affine group, directions, finite plane, growth.
\medskip

\textbf{MSC2010.} 52C10, 52C30, 20F69, 20E34, 12E10.
\end{minipage}
\bigskip

\section{Introduction}

Among the many different problems related to the study of growth and expansion in finite groups, the study of the affine group over finite fields has occupied a particularly interesting place. The affine group
\begin{equation*}
\mathrm{Aff}(\mathbb{F})=\left\{\left.\begin{pmatrix}a&b\\0&1\end{pmatrix}\right|a\in\mathbb{F}^{*},b\in\mathbb{F}\right\},
\end{equation*}
where $\mathbb{F}$ is a finite field, is one of the smallest interesting examples of an infinite family of finite groups on which questions of growth of sets $A\subseteq\mathrm{Aff}(\mathbb{F})$ can yield nontrivial answers, and it has been used to showcase techniques applicable to more general situations, like the pivot argument; on the other hand, its shape makes its uniquely suitable to study the so-called sum-product phenomenon, related to growth of sets inside finite fields under both addition and multiplication. For both of these points of view, a remarkable example is provided in Helfgott's survey \cite[\S 4.2]{Hel15}.

Structural theorems about growth in $\mathrm{Aff}(\mathbb{F}_{p})$ ($p$ prime) have been produced in the last few years, describing in substance what a set $A$ with small growth must look like. Results like Helfgott's \cite[Prop. 4.8]{Hel15} and Murphy's \cite[Thm. 27]{Mur17} belong to a first generation of proofs that rely, one way or another, on sum-product estimates; they already accomplish the goal of characterizing quite well a slowly growing $A$: such a set must essentially either be a point stabilizer or be contained in a few vertical lines, which in addition get filled in finitely many steps if $|A|\gg p$.

Rudnev and Shkredov \cite{RS18} have then quantitatively improved this classification in $\mathrm{Aff}(\mathbb{F}_{p})$: the main attractivity of their result, however, resides in the fact that, in their own words, ``the improvement [they] gain is due [...] to avoiding any explicit ties with the sum-product phenomenon, which both proofs of Helfgott and Murphy relate to'', which makes their version of the characterization of slowly growing $A$ part of a new generation of efforts. What they rely on instead is a geometric theorem by Sz\H{o}nyi \cite[Thm. 5.2]{Szo99} that gives a good lower bound on the number of directions spanned by a set of non-collinear points in the plane $\mathbb{F}_{p}^{2}$ for $p$ prime.

Following the approach by Rudnev and Shkredov, we first produce an analogous version of Sz\H{o}nyi's result for the plane $\mathbb{F}_{q}^{2}$, where $q$ is any prime power; then we use that estimate to prove a structural theorem on slowly growing sets in $\mathrm{Aff}(\mathbb{F}_{q})$ (resembling the corresponding ones for $\mathrm{Aff}(\mathbb{F}_{p})$ mentioned before), which to the best of our knowledge is the first of its kind.

\begin{center}
***
\end{center}

Throughout the paper, $p$ will always denote a prime and $q$ a power of $p$. Given a set $A$ inside the plane $\mathbb{F}^{2}$, the set of \textit{directions} spanned or determined by $A$ denotes the set
\begin{equation*}
D=\left\{\left.\frac{b'-b}{a'-a}\right|(a,b),(a',b')\in A, \ (a,b)\neq(a',b')\right\}\subseteq\mathbb{F}\cup\{\infty\},
\end{equation*}
where conventionally $\infty$ corresponds to the fraction with $a'-a=0$. We make free use of the natural identification $\mathrm{Aff}(\mathbb{F})\leftrightarrow\mathbb{F}^{*}\times\mathbb{F}$ given by
\begin{equation*}
\begin{pmatrix}a&b\\0&1\end{pmatrix}\in\mathrm{Aff}(\mathbb{F}) \ \ \ \ \ \longleftrightarrow \ \ \ \ \ (a,b)\in\mathbb{F}^{*}\times\mathbb{F},
\end{equation*}
so that we may refer to points, lines and directions even when speaking of the group $\mathrm{Aff}(\mathbb{F})$; in particular, we call $\pi:\mathrm{Aff}(\mathbb{F})\rightarrow\mathbb{F}^{*}$ the map corresponding to the projection on the first component, so that the preimage of a point through this map is a vertical line. $\mathrm{Aff}(\mathbb{F})$ acts also on $\mathbb{F}$ as $(a,b)\cdot x=ax+b$, and we think of this action when we refer to $\mathrm{Stab}(x)$ (which also looks like a line when seen in $\mathbb{F}^{2}$); finally, $U$ denotes the unipotent subgroup corresponding to $\{1\}\times\mathbb{F}$, again a vertical line.

As said before, one of the starting points of the new-style result for slowly growing sets in $\mathrm{Aff}(\mathbb{F}_{p})$ is the following bound by Sz\H{o}nyi.

\begin{theorem}\label{th:szo}
Let $p$ be a prime, and let $A\subseteq\mathbb{F}_{p}^{2}$ with $1<|A|\leq p$. Then either $A$ is contained in a line or $A$ spans $\geq\frac{|A|+3}{2}$ directions.
\end{theorem}

With that, Rudnev and Shkredov prove the following (see \cite[Thm. 5]{RS18}).

\begin{theorem}\label{th:rs}
Let $p$ be a prime and let $A\subseteq\mathrm{Aff}(\mathbb{F}_{p})\leftrightarrow\mathbb{F}_{p}^{*}\times\mathbb{F}_{p}$ with $A=A^{-1}$ and $|A^{3}|=C|A|$. Then at least one of the following is true:
\begin{enumerate}[(a)]
\item\label{th:rsline} $A\subseteq\mathrm{Stab}(x)$ for some $x\in\mathbb{F}_{p}$;
\item\label{th:rssmall} when $1<|A|\leq(1+\varepsilon)p$ for some $0<\varepsilon<1$, we have $|\pi(A)|\leq 2C^{4}$;
\item\label{th:rslarge} when $|A|>(1+\varepsilon)p$ for some $0<\varepsilon<1$, we have $|\pi(A)|\ll_{\varepsilon}\frac{1}{p}C^{3}|A|$, and in particular for $|A|>4p$ we have $|\pi(A)|\leq\frac{2}{p}C^{3}|A|$ and $A^{8}\supseteq U$.
\end{enumerate}
\end{theorem}

Sz\H{o}nyi's bound is part of a long history of applications of results about lacunary polynomials over finite fields to finite geometry: the reader interested in similar applications can check \cite{Szo99} and its bibliography.

Many results in this area can apply, with the appropriate modifications, to $\mathbb{F}_{q}$ as well. In this case, however, bounds on the number of directions spanned by a set in the finite plane appear to be messier, and understandably so: unlike in the case of $\mathbb{F}_{p}$, the number of directions determined by $A$ tends to congregate around values $\frac{|A|}{p^{i}}$ for powers $p^{i}|q$; this is due to the fact that there may exist sets with multiples of $p^{i}$ points on each line that are so well-structured that they sit in relatively few directions compared to the amount of points they have (see \cite[\S 5]{BBBSS99} for an example of this assertion when $|A|=q$).

The result we essentially use, on the number of directions spanned in $\mathbb{F}_{q}^{2}$ by some set with $1<|A|\leq q$, is due to Fancsali, Sziklai and Tak\'ats \cite[Thm. 17]{FST13}: for the lower bound they found we give here a proof that is very similar to theirs, but we also prove a different upper bound that can be more or less advantageous than theirs depending on the situation (Theorem~\ref{th:qbounds}). Used directly, the lower bound can only give us $\approx\frac{|A|}{q/p}$ directions; a tighter theorem, in the style of \cite[Thm. 1.1]{BBBSS99}, would give not only $p^{i}|p^{e}=q$, but also $i|e$ (and therefore the much better $\approx\frac{|A|}{\sqrt{q}}$ for the number of directions): \cite[Thm. 1.1]{BBBSS99} however works only for $|A|=q$, and the lack of a complete set of $q$ points is crucial in worsening the condition on the denominator $p^{i}$ during the proof.

Nevertheless, it turns out that a simple observation \textit{can} make us achieve the bound with $\sqrt{q}$ in the denominator: at its core, we use the fact that a set of points $A$ either sits on $\geq\sqrt{q}$ parallel lines or has a line with $\geq\frac{|A|}{\sqrt{q}}$ points on it. Our first main result then, playing the role of Sz\H{o}nyi's bound in \cite{RS18}, is as follows.

\begin{theorem}\label{th:maindir}
Let $q=p^{e}$ be a prime power, and let $A\subseteq\mathbb{F}_{q}^{2}$ with $1<|A|\leq q$. Then either $A$ is contained in a line or $A$ spans
\begin{enumerate}[(a)]
\item\label{th:maindireven} $>\frac{|A|}{\sqrt{q}}$ directions for $e$ even,
\item\label{th:maindirodd} $>\frac{|A|}{p^{\frac{e-1}{2}}+1}$ directions for $e$ odd.
\end{enumerate}
\end{theorem}

Observe that the theorem is only a constant away from Sz\H{o}nyi's bound when we use it for $q=p$; we add that actually the proof can be easily adjusted to yield that bound exactly: we chose not to do so in order to get a cleaner statement, with case \eqref{th:maindirodd} valid for all $e$ odd.

Using Theorem~\ref{th:maindir} and following more or less the same proof as in \cite{RS18}, we obtain our second main result, generalizing Theorem~\ref{th:rs} to any $\mathbb{F}_{q}$.

\begin{theorem}\label{th:mainaff}
Let $q=p^{e}$ be a prime power and let $A\subseteq\mathrm{Aff}(\mathbb{F}_{q})\leftrightarrow\mathbb{F}_{q}^{*}\times\mathbb{F}_{q}$ with $A=A^{-1}$ and $|A^{3}|=C|A|$. Then at least one of the following is true:
\begin{enumerate}[(a)]
\item\label{th:mainaffline} $A\subseteq\mathrm{Stab}(x)$ for some $x\in\mathbb{F}_{q}$;
\item\label{th:mainaffsmall} when $1<|A|\leq q$ we have $|\pi(A)|<(p^{\lfloor\frac{e}{2}\rfloor}+2)C^{4}$, while when $q<|A|<(3+2\sqrt{2})q$ we have $|\pi(A)|<(4+2\sqrt{2})C^{4}$;
\item\label{th:mainafflarge} when $|A|\geq(3+2\sqrt{2})q$ we have $|\pi(A)|<\frac{2}{q}C^{3}|A|$ and $A^{8}\supseteq U$.
\end{enumerate}
\end{theorem}

The statement above looks remarkably similar to Theorem~\ref{th:rs}, and is qualitatively as strong a structural theorem as in the case of $\mathrm{Aff}(\mathbb{F}_{p})$.

Let us comment however on a small difference between the two. The case of a medium-sized $A$ (i.e. $1<\frac{|A|}{q}\ll 1$) has been placed into alternative \eqref{th:rslarge} by Rudnev and Shkredov and into alternative \eqref{th:mainaffsmall} by us, essentially losing the $A^{k}\supseteq U$ implication: this has been done because the subgroup $H$ of Kneser's theorem can stifle the growth of $A$, in a way that Cauchy-Davenport could not (asking for $p$ large enough is innocuous there, but not here).

We could still use Alon's bound \cite[(4.2)]{Alo86} on the number of lines in the projective plane as done in \cite{RS18}, since it holds for $\mathbb{F}_{q}$ as well: this would give for example $|\pi(A)|<\frac{2(\sqrt{5}+1)}{(7-3\sqrt{5})q}C^{3}|A|$ for $|A|\geq\frac{\sqrt{5}+1}{2}q$ (where the maximum of $\frac{\varepsilon^{2}(1-\varepsilon)}{2(1+\varepsilon)}$ is located) and in general $|\pi(A)|\ll_{\varepsilon}\frac{1}{q}C^{3}|A|$ for $|A|\geq(1+\varepsilon)q$; then, upon using Kneser's theorem, one could either ask for $p$ large enough ($p>100$ in the first case, say, and $p\gg_{\varepsilon}1$ in general) or classify separately the sets $A$ with large $H$ (which should be possible, because having large $H=\mathrm{Stab}(A^{2})$ is a rather restrictive condition to satisfy), and an additional conclusion $A^{k}\supseteq U$ for $k\ll_{\varepsilon}1$ would be reached. It would probably be interesting to explore more deeply these medium-sized sets; however, for the purpose of obtaining a structural result like Theorem~\ref{th:mainaff} whose numerical details are of secondary relevance, we deemed to be simpler and just as effective to reduce that case to alternative \eqref{th:mainaffsmall}, especially as the observation behind our ability to do so (Lemma~\ref{le:moreq}) is very elementary.

\section{Number of directions in $\mathbb{F}_{q}^{2}$}

In the present section we prove bounds about the number of directions determined by sets of points in the plane $\mathbb{F}_{q}^{2}$, which lead eventually to Theorem~\ref{th:maindir}.

Let us start with the following simple statement: it does not concern Theorem~\ref{th:maindir}, but it will allow us in the next section to deal quickly with the sets $A$ whose size is slightly larger than $q$.

\begin{lemma}\label{le:moreq}
Any set $A\subseteq\mathbb{F}_{q}^{2}$ with $|A|>q$ spans all $q+1$ directions.
\end{lemma}

\begin{proof}
The result is immediate: by the pigeonhole principle, for any given direction, one of the $q$ parallel lines in $\mathbb{F}_{q}^{2}$ following that direction has to contain at least two points of $A$.
\end{proof}

As a complement to Lemma~\ref{le:moreq}, the following theorem deals with the number of directions spanned by sets of size at most $q$. As remarked before, a theorem of the same nature appears already in \cite{FST13}, and it is proved very similarly using the same techniques deriving from the study of lacunary polynomials.

\begin{theorem}\label{th:qbounds}
Let $q=p^{e}$ be a prime power, let $A\subseteq\mathbb{F}_{q}^{2}$ with $1<|A|\leq q$, and let $D$ be the set of directions determined by $A$. Then either $|D|=1$ (and $A$ is contained in a line), $|D|=q+1$ (and $A$ spans all directions) or there are two integers $0\leq l_{2}\leq l_{1}<e$ such that
\begin{align*}
|D| & \geq\frac{|A|-1}{p^{l_{2}}+1}+2, \\
|D| & \leq q-|A|+\max\left\{1,\frac{|A|-1-(q-|A|)\max\{0,|A|+p^{l_{1}}-q-1\}}{p^{l_{1}}-1}\right\}.
\end{align*}
\end{theorem}

A little notational comment: if $l_{1}=0$ we consider the upper bound trivial (but the lower bound becomes $\frac{|A|+3}{2}$, which is quite strong, identical to Sz\H{o}nyi's bound for $\mathbb{F}_{p}$).

Before we go to the proof, let us spend a few more words comparing this result with the one in \cite{FST13}: their bounds are written as $\frac{|A|-1}{t+1}+2\leq|D|\leq\frac{|A|-1}{s-1}$, for some appropriately defined $s,t$. The lower bound is the same as the one presented here, as $t$ and $p^{l_{2}}$ are defined in the same way. The situation for the upper bound is more interesting: we have $s\leq t=p^{l_{2}}\leq p^{l_{1}}$, because the authors define $s$ looking at the multiplicities in $H_{y}(x)$ alone (see the proof below for details) instead of the whole $x^{q}+g_{y}(x)$, which also gives a stronger geometric meaning to their $s$ than to our $l_{1}$; however, our upper bound tends to be stronger when $|A|$ is fairly close to $q$ and there is a gap between $s$ and $p^{l_{1}}$ (which can happen, as observed in \cite{FST13}).

\begin{proof}
First of all, we can suppose $\infty\in D$. If this were not true, we could take any $d\in D\setminus\{0\}$ ($D$ is nonempty for $|A|>1$, and $D=\{0\}$ concludes the theorem) and consider $A'$ made of points $(a-db,b)$ for any $(a,b)\in A$, which implies also that $|A'|=|A|$: such a set would span directions given by
\begin{equation*}
\frac{b'-b}{a'-db'-a+db}=\frac{1}{\frac{a'-a}{b'-b}-d},
\end{equation*}
from which it is clear that the new set of directions $D'$ is as large as $D$, since equalities are preserved, and that moreover $\infty\in D'$.

Define $n\geq 0$ such that $|A|=q-n$. First, define the R\'edei polynomial
\begin{equation*}
H_{y}(x)=\prod_{i=1}^{q-n}(x+ya_{i}-b_{i})\in\mathbb{F}_{q}[x,y],
\end{equation*}
where the product is among all the $(a_{i},b_{i})\in A$: it is a polynomial of degree $q-n$ in two variables (some authors, like in \cite{BBBSS99}, define it as a homogeneous polynomial in three variables, but by ensuring that $\infty \in D$ we do not need to do so). The usefulness of such polynomial lies in the fact that two points of $A$ sitting on the same line with slope $y_{0}$ yield the same $x+y_{0}a-b$, so that a multiple root in $H_{y_{0}}(x)$ reflects the presence of a line with multiple points, i.e. a secant of $A$, and indicates that $y_{0}\in D$. We also define another function in two variables,
\begin{equation}\label{eq:fdef}
f_{y}(x)=\sum_{j=0}^{n}(-1)^{j}\sigma_{j}(\mathbb{F}_{q}\setminus\{ya_{i}-b_{i}|(a_{i},b_{i})\in A\})x^{n-j},
\end{equation}
where $\sigma_{j}(S)$ is the $j$-th elementary symmetric polynomial of the elements in the set $S$; $f_{y}(x)$ is itself a polynomial in two variables (see \cite[Thm. 4]{Szo96} for a recursive definition of $f_{y}(x)$), in which the coefficient of $x^{n-j}$ has $y$-degree $j$: therefore we can write
\begin{equation*}
x^{q}+g_{y}(x)=H_{y}(x)f_{y}(x)\in\mathbb{F}_{q}[x,y],
\end{equation*}
where $g_{y}(x)$ is a polynomial in two variables of $x$-degree $\leq q-1$.

Substituting $y=y_{0}$ for some $y_{0}\not\in D$, we observe that by definition the set $\mathbb{F}_{q}\setminus\{y_{0}a_{i}-b_{i}|(a_{i},b_{i})\in A\}$ has $n$ elements and that $f_{y_{0}}(x)$ is simply the product of the $x-k_{i}$ for all the $k_{i}\in\mathbb{F}_{q}$ not counted in $H_{y_{0}}(x)$, so $g_{y_{0}}(x)=-x$: this means that the coefficients of $x^{q-1},x^{q-2},\ldots,x^{|D|}$ in $g_{y}(x)$ are polynomials of degree $\leq q-|D|$ in $y$ that take value $0$ for the $q-|D|+1$ values $y_{0}\in\mathbb{F}_{q}\setminus D$. Thus, these coefficients are the zero polynomial; in other words, the $x$-degree of $g_{y}(x)$ is at most $|D|-1$.

Working with $x,y$ has allowed us to give a bound on the degree of $g_{y}(x)$. From now on, for the sake of simplicity we substitute one value $y\in D\setminus\{\infty\}$ inside our polynomials and drop the index, and we will work with only one variable; this is possible unless $D=\{\infty\}$, from which $|D|=1$ and $A$ is contained in a vertical line (or a general line, if we got to $\infty\in D$ by the linear transformation at the beginning of the proof).

Call $l_{2}$ the largest integer for which $g(x)\in\mathbb{F}_{q}[x^{p^{l_{2}}}]$: by the fact that any $x\mapsto x^{p^{i}}$ is an automorphism of $\mathbb{F}_{q}$, we have $g(x)=(\tilde{g}(x))^{p^{l_{2}}}$ for some $\tilde{g}(x)\in\mathbb{F}_{q}[x]\setminus\mathbb{F}_{q}[x^{p}]$. Decompose $x^{q}+g(x)$ into its irreducible factors, and call $l_{1}$ the largest integer for which $p^{l_{1}}$ divides the multiplicity of each linear factor (hence $l_{1}\geq l_{2}$): $l_{1},l_{2}$ depend on our choice of $y$, so for our definition we suppose that we have chosen a $y$ that yields the smallest $l_{1}$. We can write
\begin{equation*}
x^{q/p^{l_{2}}}+\tilde{g}(x)=(R(x))^{p^{l_{1}-l_{2}}}N(x),
\end{equation*}
where $R(x)\in\mathbb{F}_{q}[x]\setminus\mathbb{F}_{q}[x^{p}]$ is such that $(R(x))^{p^{l_{1}}}$ is the divisor of $x^{q}+g(x)$ made of its linear factors (the fully reducible part of $x^{q}+g(x)$) and $N(x)\in\mathbb{F}_{q}[x]\setminus\mathbb{F}_{q}[x^{p}]$ is such that $(N(x))^{p^{l_{2}}}$ is the divisor of $x^{q}+g(x)$ made of its nonlinear factors. Note that $(N(x))^{p^{l_{2}}}$ must be a divisor of $f(x)$. If $l_{1}=e$ then $x^{q}+g(x)=(x+c)^{q}$ for some $c\in\mathbb{F}_{q}$, which means that all the points of $A$ lie on a line of slope equal to the $y$ we have fixed, contradicting $\infty\in D$: hence $l_{2}\leq l_{1}<e$.

Call $R^{*}(x)$ the divisor of $R(x)$ made of all the irreducible factors of $R(x)$ counted without multiplicity: $R^{*}(x)$ divides also $x^{q}-x$ by definition, so it divides $x^{q}+g(x)-(x^{q}-x)=g(x)+x\neq 0$ ($y\in D$ prevents us from having $g(x)=-x$). If an irreducible polynomial $k_{1}(x)$ divides another $k_{2}(x)$ with multiplicity $m$ then it divides $k_{2}'(x)$ with multiplicity $m-1$, so
\begin{equation*}
\left.\frac{(R(x))^{p^{l_{1}-l_{2}}}}{R^{*}(x)}\right|(x^{q/p^{l_{2}}}+\tilde{g}(x))'=\tilde{g}'(x)\neq 0,
\end{equation*}
where the last is true because $\tilde{g}(x)\not\in\mathbb{F}_{q}[x^{p}]$. From the reasoning above, we obtain
\begin{equation*}
x^{q}+g(x)=\left.\left(R^{*}(x)\cdot\frac{(R(x))^{p^{l_{1}-l_{2}}}}{R^{*}(x)}\right)^{p^{l_{2}}}(N(x))^{p^{l_{2}}}\right|(g(x)+x)^{p^{l_{2}}}(\tilde{g}'(x))^{p^{l_{2}}}f(x)\neq 0,
\end{equation*}
and therefore $q=\deg(x^{q}+g(x))\leq p^{l_{2}}(\deg(g(x)+x)+\deg\tilde{g}'(x))+\deg f(x)$; we have already determined that $\deg(g(x)+x)\leq\deg g(x)\leq|D|-1$, and similarly $\deg\tilde{g}'(x)\leq\frac{\deg g(x)}{p^{l_{2}}}-1\leq\frac{|D|-1}{p^{l_{2}}}-1$, hence from the definition of $f(x)$ we get
\begin{align*}
q & \leq p^{l_{2}}\left(|D|-1+\frac{|D|-1}{p^{l_{2}}}-1\right)+n & & \Longrightarrow & |D| & \geq\frac{q-n-1}{p^{l_{2}}+1}+2,
\end{align*}
settling the lower bound.

Let us focus now on the upper bound. Fix a point $(a,b)\in A$ and take a slope $y_{0}\in\mathbb{F}_{q}$: the multiplicity of the linear factor $x+y_{0}a-b$ inside $H(x)$ determines how many points of $A$ sit on the line defined by $(a,b)$ and $y_{0}$. We know that the multiplicity of every linear factor in the whole $H(x)f(x)$ is a multiple of $p^{l_{1}}$ and that it is at least $1$ for this particular linear factor, since $(a,b)$ sits on the line; however, we need a way to keep under control the number of false positives that come from the fully reducible part of $f(x)$ (inexistent ``ghost points'' that make us overcount the contribution of a single line to $A$, and thus undercount $|D|$). The way to go is to bound the number of lines passing through $(a,b)$ for which ghost points exist.

Let $f_{y}(x)$ be as in \eqref{eq:fdef}, call it for simplicity $f_{y}(x)=\sum_{j=0}^{n}\sigma_{y,j}x^{n-j}$ where the $\sigma_{y,j}$ are polynomials in $y$ of degree $j$. Assume that $|D|<q+1$: then there will be a direction $y_{0}\in\mathbb{F}_{q}\setminus D$, as $\infty\in D$. For this $y_{0}$, $H_{y_{0}}(x)f_{y_{0}}(x)=x^{q}-x$ and $x+y_{0}a-b$ has multiplicity $1$ in it; moreover, it must come from our fixed point $(a,b)$, which means that it must divide $H_{y_{0}}(x)$ and be coprime with $f_{y_{0}}(x)$: this fact implies that the two-variable linear polynomial $x+ya-b$ cannot divide $f_{y}(x)$. In other words, we cannot write
\begin{equation}\label{eq:divisor}
(x^{n-1}+\tau_{y,1}x^{n-2}+\ldots+\tau_{y,n-1})(x+ya-b)=x^{n}+\sigma_{y,1}x^{n-1}+\ldots+\sigma_{y,n}
\end{equation}
for any choice of polynomials $\tau_{y,i}$; however, defining
\begin{equation*}
\tau_{y,i}=\sum_{j=0}^{i}(-1)^{j}(ya+b)^{j}\sigma_{y,i-j}
\end{equation*}
(here $\sigma_{y,0}=1$) we can ensure that the equality \eqref{eq:divisor} works at least at the level of the coefficients of $x,x^{2},\ldots,x^{n-1}$, which means that we must have
\begin{equation}\label{eq:freecoef}
\sum_{j=0}^{n}(-1)^{j}(ya+b)^{j}\sigma_{y,n-j}\neq 0,
\end{equation}
so as to violate \eqref{eq:divisor} for the free coefficient.

Every time $f_{y_{i}}(x)$ has a $x+y_{i}a-b$ factor (or, geometrically speaking, every time the line determined by $(a,b)$ and $y_{i}$ has a ghost point), \eqref{eq:divisor} is true for $y=y_{i}$ though, and in particular the LHS of \eqref{eq:freecoef} is indeed $0$: that expression is a polynomial in $y$ of degree $n$, so if there were $n+1$ lines with ghost points \eqref{eq:freecoef} would not be true, contradicting the fact that $x+ya-b$ cannot divide $f_{y}(x)$ as polynomials in two variables. Hence, at most $n$ non-vertical lines through $(a,b)$ have ghost points.

If $|D|-1\leq n$ the upper bound stated in the theorem is already true, so suppose that the opposite holds: then there is a non-vertical line through $(a,b)$ whose slope is in $D$ with no ghosts. We can transform $A$ as at the beginning of the proof to make that slope $\infty$, i.e. $(a,b)$ lies on a vertical secant of $A$.

Each non-vertical line through $(a,b)$ whose slope is in $D$ has a multiple of $p^{l_{1}}$ among true points of $A$ and ghost points ($l_{1}$ has been defined so as to make that statement true for all slopes at the same time). On the ghost-free lines there are at least $p^{l_{1}}-1$ true points besides $(a,b)$, while on the ones with ghosts we can only say that there are at least $\max\{0,p^{l_{1}}-1-n\}$ of them (as the $x$-degree of $f_{y}(x)$ is $n$); finally, the vertical secant has at least $p^{l_{1}}$ points including $(a,b)$, as we made sure that it had no ghosts before the transformation. Combining all of this with the bound on the number of lines with ghosts, and counting all the points of $A$, we get
\begin{equation*}
(|D|-1-n)(p^{l_{1}}-1)+n\max\{0,p^{l_{1}}-1-n\}+p^{l_{1}}\leq q-n.
\end{equation*}
As we remarked after the statement of the theorem, for $l_{1}=0$ there is no upper bound. For $l_{1}>0$, the inequality above concludes the proof.
\end{proof}

Now that we have the lower bound provided by Theorem~\ref{th:qbounds}, we can proceed with the proof of the first main theorem. We retain the same notation as in the previous proof.

\begin{proof}[Proof of Thm.~\ref{th:maindir}]
Suppose that $|D|\neq 1,q+1$ (otherwise the theorem is already proved); fix a slope $y_{0}\neq\infty$ and consider the polynomial $R^{*}(x)$ defined as in the proof of Theorem~\ref{th:qbounds}. Let $\varepsilon>0$ be small enough, and let $q'=p^{\frac{e}{2}}-\varepsilon$ for $e$ even and $q'=p^{\frac{e-1}{2}}$ for $e$ odd.

If the degree of $R^{*}(x)$ is $\leq q'$, the set $A$ must be contained in $\leq q'$ lines with slope $y_{0}$, which means that one of them (call it $L$) will have to contain $\geq\frac{|A|}{q'}$ points of $A$; since $A$ is not contained in one line there must be also a point of $A$ outside $L$, and each secant laid between this point and a point of $A\cap L$ has a different slope, so that $|D|\geq\frac{|A|}{q'}$: for $e$ even it means $|D|>\frac{|A|}{\sqrt{q}}$, while for $e$ odd it means $|D|\geq\frac{|A|}{p^{\frac{e-1}{2}}}$.

If $R^{*}(x)$ has degree $>q'$, then by the fact that $(R^{*}(x))^{p^{l_{1}}}$ divides $x^{q}+g(x)$ we must have $p^{l_{2}}\leq p^{l_{1}}<\frac{q}{q'}$: regardless of whether $e$ is even or odd, $p^{l_{2}}\leq p^{\lfloor\frac{e}{2}\rfloor}$ since $l_{2}$ is an integer. Using the lower bound in Theorem~\ref{th:qbounds} (which holds for our $A$), we have
\begin{equation*}
|D|\geq\frac{|A|-1}{p^{\lfloor\frac{e}{2}\rfloor}+1}+2=\frac{|A|}{p^{\lfloor\frac{e}{2}\rfloor}}+2-\frac{|A|}{p^{\lfloor\frac{e}{2}\rfloor}(p^{\lfloor\frac{e}{2}\rfloor}+1)}-\frac{1}{p^{\lfloor\frac{e}{2}\rfloor}+1}.
\end{equation*}
For $e$ even, the bound above implies $|D|>\frac{|A|}{\sqrt{q}}$, while for $e$ odd we can obtain $|D|\geq\frac{|A|+3}{p^{\frac{e-1}{2}}+1}$.
\end{proof}

\section{Growth in $\mathrm{Aff}(\mathbb{F}_{q})$}

We move now to the proof of Theorem~\ref{th:mainaff}. We follow closely the proof of the analogous result in \cite{RS18} for $\mathbb{F}_{p}$: the only difference is that we use Theorem~\ref{th:maindir} instead of Sz\H{o}nyi's bound, and that as we have already said we absorb the case of $A$ of medium size into alternative \eqref{th:mainaffsmall}, without resorting to Alon's bound to fall into \eqref{th:mainafflarge}.

We remind the reader of two well-known and by now classical results. First, an inequality, deducible in multiple ways from bounds by Ruzsa (see for instance \cite{Ruz96}), states that for any group $G$ and any $A=A^{-1}\subseteq G$ the equality $|A^{3}|=C|A|$ implies $|A^{k}|\leq C^{k-2}|A|$ for any $k\geq 4$. Second, Kneser's theorem \cite{Kne53} tells us that, given any abelian $G$ and any $A,B\subseteq G$, there is a proper subgroup $H$ with $|A+B|\geq\min\{|G|,|A|+|B|-|H|\}$.

Theorem~\ref{th:maindir} and Lemma~\ref{le:moreq} will take care of small and medium $|A|$, respectively. For $|A|$ large we will instead make use of the following bound, due to Vinh \cite[Thm. 3]{Vin11}: the statement therein says actually something weaker, but it is based on a well-known graph-theoretic result \cite[Cor. 9.2.5]{AS16} that lets us reformulate as follows (as \cite{RS18} does for $\mathbb{F}_{p}$).

\begin{proposition}\label{pr:vinh}
Let $q$ be a prime power, let $P$ be a set of points in $\mathbb{F}_{q}^{2}$ and let $L$ be a set of lines in $\mathbb{F}_{q}^{2}$; define $I(P,L)$ as the set of pairs $(p,l)\in P\times L$ s.t. $p\in l$. Then
\begin{equation*}
\left|I(P,L)-\frac{|P||L|}{q}\right|\leq\sqrt{q|P||L|}.
\end{equation*}
\end{proposition}

Let us also give here separately a lemma that will provide the upper bounds on $\pi(A)$ in Theorem~\ref{th:mainaff}\eqref{th:mainaffsmall}-\eqref{th:mainafflarge}.

\begin{lemma}\label{le:phipi}
For any $g=\begin{pmatrix}a&b\\0&1\end{pmatrix}\in\mathrm{Aff}(\mathbb{F}_q)\setminus\{\mathrm{Id}\}$, define the map
\begin{equation*}
\varphi_{g}:\mathrm{Aff}(\mathbb{F}_q)\rightarrow\mathrm{Aff}(\mathbb{F}_q), \ \ \ \ \ \varphi_{g}(h)=hgh^{-1}.
\end{equation*}
Then,
\begin{enumerate}[(a)]
\item\label{le:phipislope} any point in the image of $\varphi_{g}$ has as preimage a line of slope $\frac{b}{a-1}$;
\item\label{le:phipibound} if $A=A^{-1}\subseteq\mathrm{Aff}(\mathbb{F}_{q})$ and $g\in A^{k}$ then $|\pi(A)|\leq\frac{|A^{k+3}|}{|\varphi_{g}(A)|}$.
\end{enumerate}
\end{lemma}

\begin{proof}
\eqref{le:phipislope} This is just an easy computation: as
\begin{equation*}
\begin{pmatrix}r&s\\0&1\end{pmatrix}\begin{pmatrix}a&b\\0&1\end{pmatrix}\begin{pmatrix}r^{-1}&-r^{-1}s\\0&1\end{pmatrix}=\begin{pmatrix}a&br+(1-a)s\\0&1\end{pmatrix},
\end{equation*}
two elements are in the preimage of a single point if and only if $br+(1-a)s=br'+(1-a)s'$, from which all pairs of elements with $\frac{s'-s}{r'-r}=\frac{b}{a-1}$ must be sent to the same point by $\varphi_{g}$ (we allow $a=1$ and a slope equal to $\infty$, but we avoid $(a,b)=(1,0)$ since $g\neq\mathrm{Id}$).

\eqref{le:phipibound} On one hand we have $|A\varphi_{g}(A)g^{-1}|=|A\varphi_{g}(A)|\leq|A^{k+3}|$, while on the other hand any element of $A\varphi_{g}(A)g^{-1}$ is of the form $a_{1}a_{2}ga_{2}^{-1}g^{-1}\in AU$: since
\begin{equation*}
\begin{pmatrix}x&y\\0&1\end{pmatrix}\begin{pmatrix}1&z\\0&1\end{pmatrix}=\begin{pmatrix}x&xz+y\\0&1\end{pmatrix},
\end{equation*}
pairs in $A\times U$ with either distinct $x$ or with the same $x,y$ but distinct $z$ will all give different products in $AU$; hence we can select one element of $A$ for each value of $x$ (therefore $|\pi(A)|$ of them) and all the $a_{2}ga_{2}^{-1}g^{-1}$ ($|\varphi_{g}(A)|$ of them, they are all multiplied by the same $g^{-1}$) and obtain the other side of the bound, namely $|A\varphi_{g}(A)g^{-1}|\geq|\pi(A)||\varphi_{g}(A)|$.
\end{proof}

With these tools at our disposal, we can proceed with the proof.

\begin{proof}[Proof of Thm.~\ref{th:mainaff}]
Let us start with the case of $A$ large: impose $|A|\geq cq$ for a constant $c>1$ to be chosen later.

We use the bound from Proposition~\ref{pr:vinh} with $P=A$ and $L=\overline{L(A)}$ (the set of lines that are not determined by $A$), interpreted as a lower bound on the expression inside the absolute value, and combine it with the trivial observation that $I(A,\overline{L(A)})\leq|\overline{L(A)}|$ by the definition of $\overline{L(A)}$: this yields
\begin{align*}
|\overline{L(A)}| & \leq q^{2}\frac{c}{(c-1)^{2}} & & \Longrightarrow & |L(A)| & \geq q+q^{2}\left(1-\frac{c}{(c-1)^{2}}\right).
\end{align*}
If $c\geq 1+\frac{1+\sqrt{3-\frac{2}{p}}}{1-\frac{1}{p}}$ (or $c\geq 3+2\sqrt{2}$, which is an upper bound for all primes $p$), by the pigeonhole principle there must exist $\geq\frac{q}{2}\left(1+\frac{1}{p}\right)$ non-vertical lines of $L(A)$ parallel to each other; call $d$ the direction defined by such lines. Given any two elements of $A$ sitting on one of these lines, we have
\begin{equation*}
g=\begin{pmatrix}a_{1}&b_{1}\\0&1\end{pmatrix}^{-1}\begin{pmatrix}a_{2}&b_{2}\\0&1\end{pmatrix}=\begin{pmatrix}a_{2}a_{1}^{-1}&b_{2}a_{1}^{-1}-b_{1}a_{1}^{-1}\\0&1\end{pmatrix}=\begin{pmatrix}a'&b'\\0&1\end{pmatrix},
\end{equation*}
with $\frac{b'}{a'-1}=\frac{b_{2}-b_{1}}{a_{2}-a_{1}}=d$, so by Lemma~\ref{le:phipi}\eqref{le:phipislope} there are $\geq\frac{q}{2}\left(1+\frac{1}{p}\right)>\frac{q}{2}$ elements in $\varphi_{g}(A)$; by Lemma~\ref{le:phipi}\eqref{le:phipibound} and Ruzsa's inequality, this implies that $|\pi(A)|<\frac{2|A^{5}|}{q}\leq\frac{2}{q}C^{3}|A|$. Moreover, the unipotent subgroup $U$ is isomorphic to $\mathbb{F}_{q}$ as an additive group, so that its largest proper subgroup is of size $\frac{q}{p}$; therefore, since $\varphi_{g}(A)g^{-1}\subseteq A^{6}\cap U$ has $|\varphi_{g}(A)g^{-1}|\geq\frac{q}{2}\left(1+\frac{1}{p}\right)$, by Kneser's theorem we must have
\begin{equation*}
A^{8}\supseteq AgAAg^{-1}A\supseteq(\varphi_{g}(A)g^{-1})(\varphi_{g}(A)g^{-1})^{-1}\supseteq U,
\end{equation*}
and we fall into case \eqref{th:mainafflarge} of the theorem.

Let us deal now with $A$ of medium size: suppose $q<|A|<cq$, so that by Lemma~\ref{le:moreq} every direction is determined by some pair of points of $A$. Partition $A^{2}\setminus\{\mathrm{Id}\}$ into $q+1$ subsets, collecting into each one of them elements having the same value for $\frac{b}{a-1}\in\mathbb{F}_{q}\cup\{\infty\}$. Every two distinct $a_{1},a_{2}\in A$ yield an element $a_{1}^{-1}a_{2}\in A^{2}$ that is located into the part corresponding to the slope of the line they define: by the pigeonhole principle there will be a part (identifiable with some $d\in\mathbb{F}_{q}\cup\{\infty\}$) with at most $\frac{|A^{2}|-1}{q+1}$ elements, and therefore every line of $L(A)$ in the direction $d$ must have at most $\frac{|A^{2}|-1}{q+1}+1$ elements of $A$ on it, since $a_{1}^{-1}a_{i}\neq a_{1}^{-1}a_{j}$ for $a_{i}\neq a_{j}$. We have thus given a bound on the number of points of $A$ sent to the same element by the map $\varphi_{g}$ for some $g\in A^{2}$ with $\frac{b}{a-1}=d$, which translates to
\begin{equation*}
|\varphi_{g}(A)|\geq\frac{(q+1)|A|}{|A^{2}|+q}>\frac{|A|}{Cc+1};
\end{equation*}
proceeding as before, by Lemma~\ref{le:phipi}\eqref{le:phipibound} and Ruzsa's inequality we conclude that $|\pi(A)|<(Cc+1)C^{3}\leq(4+2\sqrt{2})C^{4}$ and we reach case \eqref{th:mainaffsmall} of the theorem.

For $A$ small (i.e. $1<|A|\leq q$) we repeat essentially what we did for $A$ medium, but instead of $|D|=q+1$ we use the bounds in Theorem~\ref{th:maindir} on the number of directions $|D|$ spanned by $A$. We obtain
\begin{equation*}
|\varphi_{g}(A)|>\frac{|A|^{2}}{q'(|A^{2}|-1)+|A|}>\frac{|A|}{Cq'+1},
\end{equation*}
where $q'=\sqrt{q}$ for $e$ even and $q'=p^{\frac{e-1}{2}}+1$ for $e$ odd, from which we get $|\pi(A)|<(p^{\lfloor\frac{e}{2}\rfloor}+2)C^{4}$ and end up in case \eqref{th:mainaffsmall}. Finally, we need to deal with the other alternative in Theorem~\ref{th:maindir}, namely that $A$ may be contained in one line: in other words, the elements of $A$ are either all of the form $(a,ad+b)$ for some $b,d\in\mathbb{F}_{q}$, through the identification of $\mathrm{Aff}(\mathbb{F}_{q})$ with $\mathbb{F}_{q}^{*}\times\mathbb{F}_{q}$, or all contained in $U$. In the latter alternative $A\subseteq U$ implies $|\pi(A)|=1$, yielding \eqref{th:mainaffsmall}; in the former, since $A=A^{-1}$ and $(a,ad+b)^{-1}=(a^{-1},-a^{-1}b-d)$, we must have $b=-d$ and then $A\subseteq\mathrm{Stab}(-d)$.
\end{proof}

\section*{Acknowledgements}

The author thanks H. A. Helfgott for attracting his attention to the paper \cite{RS18}. The author also thanks K. M\"uller and I. D. Shkredov for helpful remarks.

\bibliography{Bibliography}

\newcommand{\etalchar}[1]{$^{#1}$}
\begin{thebibliography}{BBB{\etalchar{+}}99}

\bibitem[Alo86]{Alo86}
N.~Alon.
\newblock Eigenvalues, geometric expanders, sorting in rounds and {R}amsey
  theory.
\newblock {\em Combinatorica}, 6(3):207--219, 1986.

\bibitem[AS16]{AS16}
N.~Alon and J.~H. Spencer.
\newblock {\em The probabilistic method}.
\newblock Wiley Publishing, fourth edition, 2016.

\bibitem[BBB{\etalchar{+}}99]{BBBSS99}
A.~Blokhuis, S.~Ball, A.~E. Brouwer, L.~Storme, and T.~Sz\H{o}nyi.
\newblock On the number of slopes of the graph of a function defined on a
  finite field.
\newblock {\em J. Combin. Theory Ser. A}, 86:187--196, 1999.

\bibitem[FST13]{FST13}
S.~L. Fancsali, P.~Sziklai, and M.~Takáts.
\newblock The number of directions determined by less than $q$ points.
\newblock {\em J. Algebraic Combin.}, 37:27--37, 2013.

\bibitem[Hel15]{Hel15}
H.~A. Helfgott.
\newblock Growth in groups: ideas and perspectives.
\newblock {\em Bull. Amer. Math. Soc. (N.S.)}, 52(3):357--413, 2015.

\bibitem[Kne53]{Kne53}
M.~Kneser.
\newblock {Absch\"atzung der asymptotischen Dichte von Summenmengen}.
\newblock {\em Math. Z.}, 58:459--484, 1953.
\newblock In German.

\bibitem[Mur17]{Mur17}
B.~Murphy.
\newblock Upper and lower bounds for rich lines in grids.
\newblock \texttt{arXiv:1709.10438}, 2017.

\bibitem[RS18]{RS18}
M.~Rudnev and I.~D. Shkredov.
\newblock On growth rate in $\text{SL}_{2}(\mathbb{F}_{p})$, the affine group
  and sum-product type implications.
\newblock \texttt{arXiv:1812.01671}, 2018.

\bibitem[Ruz96]{Ruz96}
I.~Z. Ruzsa.
\newblock Sums of finite sets.
\newblock In D.~V. Chudnovsky, G.~V. Chudnovsky, and M.~B. Nathanson, editors,
  {\em Number Theory: {N}ew {Y}ork {S}eminar 1991-1995}, pages 281--293.
  Springer-Verlag, New York (USA), 1996.

\bibitem[Sz{\H{o}}96]{Szo96}
T.~Sz{\H{o}}nyi.
\newblock On the number of directions determined by a set of points in an
  affine {G}alois plane.
\newblock {\em J. Combin. Theory Ser. A}, 74:141--146, 1996.

\bibitem[Sz{\H{o}}99]{Szo99}
T.~Sz{\H{o}}nyi.
\newblock Around {R}\'edei's theorem.
\newblock {\em Discrete Math.}, 208/209:557--575, 1999.

\bibitem[Vin11]{Vin11}
L.~A. Vinh.
\newblock The {S}zemer\'edi-{T}rotter type theorem and the sum-product estimate
  in finite fields.
\newblock {\em European J. Combin.}, 32:1177--1181, 2011.

\end{thebibliography}
\bibliographystyle{alpha}

\end{document}